\documentclass[12pt]{article}
\usepackage{amssymb}
\usepackage{amsmath,amsthm}
\usepackage[latin1]{inputenc}
\usepackage[dvips]{graphicx}
\usepackage{hyperref}
\usepackage{enumerate}
\usepackage{color}
\usepackage{tikz,ifthen}
\usepackage[T1]{fontenc}

\hypersetup{colorlinks=true}
\hypersetup{colorlinks=true, linkcolor=blue, citecolor=blue,urlcolor=blue}

\newtheorem{remark}{Remark}

\newtheorem{definition}{Definition}

\newtheorem{lemma}[remark]{Lemma}
\newtheorem{theorem}[remark]{Theorem}
\newtheorem{proposition}[remark]{Proposition}
\newtheorem{corollary}[remark]{Corollary}
\newtheorem{claim}[remark]{Claim}

\newcommand{\smallqed}{{\tiny ($\Box$)}}
\begin{document}

\title{Incidence dimension and 2-packing number in graphs}
\author{Dragana Bo\v{z}ovi\'c$^{(1)}$, Aleksander Kelenc$^{(2,3,5)}$, Iztok Peterin$%
^{(1,3)}$, Ismael G. Yero$^{(4)}$ \\
\\
$^{(1)}$ {\small Faculty of Electrical Engineering and Computer Science}\\
{\small University of Maribor,} {\small Koro\v{s}ka cesta 46, 2000 Maribor,
Slovenia.} \\
$^{(2)}$ {\small Faculty of Natural Sciences and Mathematics}\\
{\small University of Maribor,} {\small Koro\v{s}ka cesta 160, 2000 Maribor,
Slovenia.} \\
$^{(3)}${\small Institute of Mathematics, Physics and Mechanics}\\
{\small Jadranska ulica 19, 1000 Ljubljana, Slovenia.} \\
$^{(4)}${\small Departamento de Matem\'aticas, Escuela Polit\'ecnica
Superior de Algeciras}\\
{\small Universidad de C\'adiz,} {\small Av. Ram\'on Puyol s/n, 11202
Algeciras, Spain.} \\
$^{(5)}${\small Center for Applied Mathematics and Theoretical Physics}\\
{\small University of Maribor,} {\small Mladinska 3, 2000 Maribor,
Slovenia.} \\
{\small \texttt{e-mails:}\textit{dragana.bozovic\@@um.si, aleksander.kelenc\@@um.si, iztok.peterin\@@um.si, ismael.gonzalez\@@uca.es}} \\
}
\date{}
\maketitle

\begin{abstract}
Let $G=(V,E)$ be a graph. A set of vertices $A$ is an incidence generator for $G$ if for any two distinct edges $e,f\in E(G)$ there exists a vertex from $A$ which is an endpoint of either $e$ or $f$. The smallest cardinality of an incidence generator for $G$ is called the incidence dimension and is denoted by $dim_I(G)$. A set of vertices $P$ is a 2-packing if the distance between any pair of distinct vertices from $P$ is greater than two. The largest cardinality of a 2-packing of $G$ is the packing number of $G$ and is denoted by $\rho(G)$. The incidence dimension of graphs is introduced and studied in this article, and we emphasize in the closed relationship between $dim_I(G)$ and $\rho(G)$. We first note that the complement of any 2-packing in a graph $G$ is always an incidence generator for $G$, and further show that either $dim_I(G)=\rho(G)$ or $dim_I(G)=\rho(G)-1$ for any graph $G$. In addition, we also prove that the problem of determining the incidence dimension of a graph is NP-complete, and present some bounds for it.
\end{abstract}



\section{Introduction}

The famous Gallai's theorem states that
\begin{equation*}
\alpha(G)+\beta (G)=n,
\end{equation*}
where $G$ is a graph on $n$ vertices, and $\alpha(G)$ and $\beta(G)$
represent the vertex cover number and the independence number, respectively, of $G$. Its
beauty lies not only in the numbers, but also in the fact that the union of any
vertex cover set and an independent set that yield $\alpha(G)$ and $\beta(G)$,
respectively, results in the whole vertex set. With such an elegance, the
hunt for analog results is always open.

One possibility occurs if we observe an independent set from an equivalent,
but yet different, perspective. This comes from the notion of $k$-packings in graphs. A set $P\subseteq V(G)$ is a $k$-\textit{packing set} (or $k$-packing for short) of $G$ if the distance between any pair of distinct vertices from $P$ is greater than $k$. The $k$-\emph{packing number}
of $G$ is the maximum cardinality of any $k$-packing of $G$ and is denoted by $\rho_k(G)$. Clearly, 1-packings represent independent sets of $G$, and maximum 1-packings are maximum independent sets of $G$. A $\rho_k(G)$-set is a packing of cardinality $\rho_k(G)$. Since we are interested only in 2-packings, hereinafter we will simply use the terminology $\rho(G)$ and packing, instead of $\rho_2(G)$ and 2-packing, respectively. A natural question concerns finding an analogy to the Gallai's theorem for maximum $k$-packings for $k\geq 2$. We precisely deal with this problem for the case when $k=2$.

Packings were an interesting topics for a longer period as a natural lower bound for the domination number $\gamma(G)$ of graphs (for definitions, terminology and more information on domination in graphs we suggest the books \cite{hhs1,hhs2}). One of the first results (and indeed very remarkable) of that type is from Meir and Moon \cite{MeMo}, who have shown that $\rho(T)=\gamma(T)$ for every tree $T$ (in a different notation). Efficient closed domination graphs represent a class of graphs with $\rho(G)=\gamma(G)$, where both maximum 2-packing sets and minimum dominating sets coincide. In such a case we call a minimum dominating set a 1-perfect code. The study of perfect codes in graphs was initiated by Biggs~\cite{biggs-1973}. Later it was intensively studied and we recommend \cite{KPY} for further information and references.

In the last decade the packing number became itself more interesting, and not only in connection with the domination number. For instance, the relationship between the packing number and maximal packings of minimum cardinality, called also the lower packing number, is investigated in \cite{SHSa}. A connection between the packing number and the double domination in the form of an upper bound is presented in \cite{MSKG}. Graphs for which their packing number equals the packing number of their complement are described in \cite{Dutt}. In \cite{HeLoRa}, it was shown that the domination number can be also bounded from above by the packing number multiplied with the maximum degree of a graph.

A generalization of packings presented in \cite{GaGuHa} is that of $k$-limited packings, where every vertex can have at most $k$ neighbors in a $k$-limited packing set $S$. A probabilistic approach to $k$-limited packings to achieve some bounds can be found in \cite{GaZv}. A further generalization, that is, generalized limited packing of the $k$-limited packing, see \cite{DoHiLe}, brings a dynamic approach with respect to the vertices of $G$, where different vertices can have a different number of neighbors in a generalized limited packing. The problem of computing the packing number of graphs is NP-hard, but polynomially solvable for $P_4$-tidy graphs as shown in \cite{DoHiLe}.

It is now our goal to continue finding several contributions on packings in connection with other topic in graphs, which has attracted the attention of several researchers in the last recent years. This is the case of the metric dimension related parameters. The concept of metric dimension in graphs (introduced first in \cite{Harary1976,Slater1975}), and its large number of variants are nowadays commonly studied, due to their properties of uniquely recognizing (identifying or determining) the vertices or edges of graphs. Some of the most recent variants precisely deals with uniquely identifying the edges of graphs (which is in some way one of the ideas which raised up this contributions). The first works on these recent topics are \cite{mix-dim,KeTrYe}. Another metric parameter that can be taken as a predecessor of that we study here concerns identification of vertices throughout neighborhoods (see \cite{JanOmo2012}). We next describe all of these related concepts.

Given a graph $G$, a set $S$ of vertices of $G$ is an \emph{adjacency
generator}\footnote{In fact, these sets were called adjacency resolving sets in \cite{JanOmo2012}%
, where the concept was first described.} for $G$ if for any two different
vertices $u,v\in V(G)-S$ there exists $x\in S$ such that $|N(x)\cap \{u,v\}|
=1$. An adjacency generator of minimum cardinality is called an \emph{adjacency basis} for $G$ and its cardinality, the \emph{adjacency dimension}
of $G$, which is denoted by $\mathrm{dim}_A(G)$. These concepts were
first introduced in \cite{JanOmo2012} as a tool while studying some metric
properties of the lexicographic product of graphs. More results on the
adjacency dimension of graphs can be found in \cite{adjacency1,adjacency2}.

Now, given a vertex $v\in V$ and an edge $e=uw\in E$, the distance between the vertex $v$ and the edge $e$ is defined as $d_G(e,v)=\min\{d_G(u,v),d_G(w,v)\}$. A vertex $w\in V$ \emph{distinguishes} two edges $e_1,e_2\in E$ if $d_G(w,e_1)\ne d_G(w,e_2)$. A nonempty set $S\subset V$ is an \emph{edge metric generator} for $G$ if any two edges of $G$ are distinguished by some vertex of $S$. An edge metric generator with the smallest possible cardinality is called an \emph{edge metric basis} for $G$, and its cardinality is the \emph{edge metric dimension}, which is denoted by $\mathrm{edim}(G)$. This concept was introduced in \cite{KeTrYe}. Some other studies on the edge metric dimension of graphs appeared in \cite{geneson,kratica,peterin-yero,zubri}.

As a kind of a mixed point of view of these two parameters above, we introduce the concept of incidence dimension in graphs which arises from the two concepts above in some natural way of research evolution. However, we shall formally define it in the next section, based on the existence of some properties of the complement of a packing set, which also yields the definition of the incidence dimension. There we also show the formal
connection between the incidence dimension and packing number of graphs. A section
about complexity of incidence dimension will further follows. We conclude this work with some
additional information about incidence dimension.

We consider only finite undirected simple graphs. Let $G$ be a graph with
vertex set $V(G)$ and edge set $E(G)$. For a fixed $v\in V(G)$, set $\{u\in
V(G):uv\in E(G)\}$ represents the \emph{open neighborhood} of $v$ and is
denoted by $N(v)$. The \emph{degree} of $v$ is $d(v) = |N(v)|$. The \emph{%
closed neighborhood} of $v\in V(G)$ is $N[v]=N(v)\cup \{v\}$. The \emph{%
distance} $d(u,v)$ between any two vertices $u$ and $v$ is the minimum
number of edges on a path between them. Given a set of vertices $S$ of $G$,
we use $G-S$ to denote the graph obtained from $G$ by removing all the
vertices of $S$ and the edges incident with them. If $S=\{v\}$ for some
vertex $v$, then we simply write $G-v$. Also, the subgraph of $G$ induced by
$D \subset V(G)$ will be denoted by $G[D]$.


\section{Defining the incidence dimension and its connection with packing number}

As mentioned in the introduction we are interested in some properties of
the complement of the packing sets of a graph $G$. The following result provides a
motivation for the definition of incidence dimension.

\begin{proposition}\label{lemma-complement}
If a set $X \subseteq V(G)$ is a packing set of a graph $G$, then the set $S=V(G)\setminus X$ is a vertex cover of $G$ and for any
two different edges $e$ and $f$ there exists $x\in S$ such that either $x$
is incident with $e$ or $x$ is incident with $f$.
\end{proposition}

\begin{proof}
First, let $e=uv$ be an arbitrary edge of $G$. If $\{u,v\} \cap S = \emptyset$, then $u,v \in X$. Since $d(u,v)=1$, this yields a contradiction with $X$ being a packing set of $G$. Thus, $S$ is a vertex cover. (This also follows from the fact that every packing is also a 1-packing and with this an independent set. Since the complements of independent sets are vertex covers the result follows.)

\noindent
Take now two different arbitrary edges $e,f \in E(G)$. If $e=uv$ and $f=ab$ are not incident, then they are distinguished by one endpoint of $e$ which exists in $S$ because $S$ is a vertex cover. Otherwise they are incident in one vertex, say in $u=a$. If $\{v,b\} \cap S = \emptyset$, then $v,b \in X$. This is a contradiction with $X$ being a packing of $G$ as $d(b,v)\leq 2$. So at least one of $v$ or $b$, say $v$, is in $S$ and $v$ is the desired vertex.
\end{proof}

With the second property of the proposition above we are able to define the
incidence dimension as follows. For this, note that it is meaningful to
demand a minimum cardinality set with this mentioned property in order to retain
the analogy with the relationship between independence number and vertex cover
number.

\begin{definition}
Given two edges $e,f\in E(G)$ and a vertex $x\in V(G)$, we say that $x$
\emph{(incidently) resolves} or \emph{distinguishes} the pair $e,f$, if
either $x\in e$ or $x\in f$ (exactly one of these two edges is incident with
$x$). A set $S$ of vertices of $G$ is an \emph{incidence generator} for $G$
if for any two different edges $e,f\in E(G)$ there exists a vertex $x\in S$
such that $x$ incidently resolves the pair $e,f$. An incidence generator
of minimum cardinality is called an \emph{incidence basis} for $G$ and its
cardinality, the \emph{incidence dimension} of $G$, is denoted by $\mathrm{dim}_I(G)$.
\end{definition}

From this we can immediately see that we cannot expect such an elegant
result as in the case of Gallai's theorem. Indeed, already in the case of $K_2$
we can see that, since there exists only one edge in $K_2$, an empty set is
an incidence generator of minimum cardinality, and we have $\mathrm{dim}_I(K_2)=0$. Clearly, this can be extended to any graph with only one edge.
However, as soon as there are two edges in $G$, we have $\mathrm{dim}_I(K_2)>0$.

The next observation is that, if $S$ is an incidence generator for $G$, then
there exists at most one edge with both endpoints outside of $S$. Namely,
such two edges would not be incidently resolved by $S$, a contradiction with $S$ being an incidence generator for $G$.

Before we state a deeper connection between the incidence dimension and the packing number of a graph we need some additional terminology.

\begin{definition}\label{definition-e-critical-packing}
Let $e=uv$ be an edge of a graph $G$. The $e$-\emph{critical packing} of $G-e$, denoted by $P_e(G)$, is a maximum packing of the graph $G-e$ with the following property.
\begin{equation}\label{condition}
\text{If } \vert \{ u, v \} \cap P_e(G) \vert < 2, \text{ then } P_e(G) \text{ is a packing of } G.
\end{equation}
\end{definition}

\noindent
Notice that both $u$ and $v$ can be in $P_e(G)$ and then (\ref{condition}) is trivially fulfilled. Clearly, $P_e(G)$ is not a packing of $G$ in such a case. However, by removing either $u$ or $v$ from $P_e(G)$ we obtain a packing of $G$.
If $u,v \notin P_e(G)$, then the set $P_e(G)$ is also a packing for $G$.
If exactly one endpoint of the edge $e$, let say $u$, is inside the set $P_e(G)$, then  $ N_G(v) \cap P_e(G) = \{u\}$ because otherwise $P_e(G)$ is not a packing of  $G$ contradicting (\ref{condition}). Therefore we have

\begin{equation}\label{boundA}
\rho(G)\leq\vert P_e(G) \vert \leq  \rho(G)+1.
\end{equation}

\noindent
Figure \ref{figure-critical-packing} shows an example of a graph $G$ where there is an $e$-critical packing smaller than $\rho(G-e)$ for the drawn dashed edge. Black vertices represent unique maximum packing of $G-e$ which does not fulfill (\ref{condition}) and is therefore not $e$-critical. Hence, the cardinality of every $e$-critical packing is two.

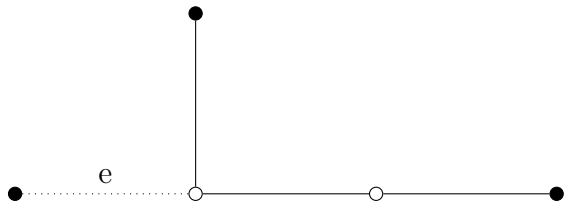
\begin{figure}[!ht]\label{figure-critical-packing}
  \begin{center}
    \begin{tikzpicture}[scale=1.2]
		\tikzstyle{rn}=[circle,fill=white,draw, inner sep=0pt, minimum size=5pt]
		\tikzstyle{bn}=[circle,fill=black,draw, inner sep=0pt, minimum size=5pt]		
		
		\node [style=bn] (0) at (0, 0) {};
		\node [style=rn] (1) at (2, 0) {};
		\node [style=rn] (2) at (4, 0) {};
		\node [style=bn] (3) at (6, 0) {};
		\node [style=bn] (4) at (2, 2) {};
		
		\draw[dotted] (0)--(1) node [midway, above] {e};
		\draw (1)--(2);
		\draw (2)--(3);
		\draw (1)--(4);

	\end{tikzpicture}
    \caption{A graph $G$ and the edge $e$ for which it holds that $\rho(G-e) > \vert P_e(G) \vert$.}
    \label{Trees}
  \end{center}
\end{figure}

\begin{theorem} \label{theorem:n-k}
If $G$ is a graph of order $n$ and $k$ is an integer defined as $\displaystyle k=\max_{e \in E(G)}{\vert P_e(G) \vert}$, then $\dim_I(G)= n-k$. An incidence basis for the graph $G$ is any set $S=V(G)\setminus P_e(G)$ for which it holds that $k=\vert P_e(G) \vert$.
\end{theorem}

\begin{proof}
Let $e=uv$ and $T=P_e(G)$ such that $k=\vert P_e(G) \vert$.
We want to prove that $S = V(G) \setminus T$ is an incidence generator for $G$.

If $\vert \{ u, v \} \cap T \vert < 2$, then $T$ is also a packing in $G$, and due to Proposition \ref{lemma-complement}, $S$ is an incidence generator for $G$. Otherwise $u,v \in T$ and so, $T$ is a packing of the graph $G-e$. Thus, $S$ is an incidence generator for $G-e$ with the property that every edge of $G-e$ has at least one endpoint in $S$, due to Proposition \ref{lemma-complement}.
Since $S$ is an incidence generator for $G-e$, we have to focus only on pairs of edges where one edge is $e$.
Consider the edge $e$ and an arbitrary edge $f \neq e$ of $G$. Clearly, the edges $e$ and $f$ are distinguished by the endpoint of $f$ that is in $S$.
It follows that $S$ is an incidence generator for $G$ and  $\mathrm{dim}_I(G) \leq n-k$.

\noindent
Now suppose that there exists an incidence generator $S'$ for $G$ with cardinality $\vert S' \vert=d' < n-k$. Since $S'$ is an incidence generator for $G$, it follows that there exists at most one edge in $G$ induced by $P_e'(G)=V(G) \setminus S'$.
Suppose that such edge $e=uv$ exists. First, notice that there is no edge in $G-e$ between two vertices of $P_e'(G)$. Also, there are not two arbitrary vertices $x,y \in P_e'(G)$ such that $d_{G-e}(x,y)=2$, since $S'$ is an incidence generator for $G$. Thus, it follows that $P_e'(G)$ is a packing for $G-e$ and both $u,v \in P_e'(G)$. Since the cardinality of $P_e'(G)$ is  $n-d' > k$, we obtain a contradiction with the maximality of $P_e(G)$.
If there is no edge in the graph induced by $P_e'(G)$, then $P_e'(G)$ is a packing of $G$, since there are no vertices at distance 2 in $P_e'(G)$. Every packing of $G$ is also a packing of $G-e$, with the property given in Definition \ref{definition-e-critical-packing}, for an arbitrary edge $e$. Again, this is a contradiction with the maximality of $P_e(G)$. Therefore, we deduce that there is no incidence generator with cardinality less than $n-k$.
\end{proof}

\noindent
A direct consequence of Theorem \ref{theorem:n-k}, together with (\ref{boundA}), is as follows.

\begin{corollary}
For every graph $G$ of order $n$ it holds that $n-\rho(G)-1\leq \mathrm{dim}_I(G)\leq n-\rho(G)$.
\end{corollary}

\noindent
This yields a natural partition of graphs into two classes, those whose incidence dimension equals to $|V(G)|-\rho(G)-1$ and those for which $\mathrm{dim}_I(G)=|V(G)|-\rho(G)$. To show that a graph $G$ belongs to the first class, we 'only' need to find an edge $e$ such that $|P_e(G)|=\rho(G)+1$. For the second class, on the other hand, we need to show that for each edge $e$ we have $|P_e(G)|=\rho(G)$.

We next derive exact results for the incidence dimension for some graph families, and we first consider a class of graphs called edge-triangular.
A graph is called \textit{edge-triangular} if every edge of a graph is in at least one 3-cycle.

\begin{proposition}\label{proposition:edge-triangular}
Let $S$ be any incidence generator for a
graph $G$. Then the graph $G$ is edge-triangular if and only if for every $e=uv \in
E(G)$ it holds that $\vert \{u,v\} \cap S \vert > 0$.
\end{proposition}

\begin{proof}
Let $G$ be an edge-triangular graph.
Suppose, that there exists en edge $e=uv$ and an incidence generator $S$ for $G$ such that $\vert \{u,v\} \cap S \vert = 0$. Since $G$ is an edge-triangular graph, there exists a vertex $w$ such that $uvwu$ is a triangle. Note that the vertex $w$ has to be in $S$ because $e$ and $uw$ have to be distinguished by at least one endpoint. But then, the edges $uw$ and $vw$ are not distinguished by any vertex from $S$. A contradiction with $S$ being an incidence generator.

Conversely, suppose that the graph $G$ is not edge-triangular. Thus, there exists an edge $e=uv$ that is not a part of a triangle. Consequently, the set $S=V(G) \setminus \{u,v\}$ is an incidence generator for $G$, which means there exists an edge $e=uv$ such that $\vert \{u,v\} \cap S \vert = 0$.
\end{proof}

Proposition \ref{proposition:edge-triangular} implies the following corollary.

\begin{corollary}\label{corollary:edge-triangular}
Let $G$ be an edge-triangular graph. The set $S$ is an incidence generator for $G$ if and only if $V(G)
\setminus S$ is a packing of $G$.
\end{corollary}

\begin{proof}
Suppose that $S$ is an incidence generator.
Due to Proposition \ref{proposition:edge-triangular} there are no two vertices at distance 1 in $V(G) \setminus S$. Since $S$ is an incidence generator, there are also not two vertices at distance 2 in $V \setminus S$. In consequence, it follows that $V(G) \setminus S$ is a packing of $G$.

\noindent
The converse holds for any graph $G$ by Proposition \ref{lemma-complement}.
\end{proof}

\begin{proposition}\label{basic-values}
Let $n,r$ and $t$ be integers.
\begin{enumerate}

\item[{\rm (i)}] If $n\ge 3$, then $\mathrm{dim}_I(K_n)=n-1$.

\item[{\rm (ii)}] If $n\ge 3$, then $\mathrm{dim}_I(P_n)=\left\lfloor\frac{2(n-1)}{3}%
\right\rfloor$.

\item[{\rm (iii)}] If $n\ge 4$, then $\mathrm{dim}_I(C_n)=\left\lfloor\frac{2n}{3}%
\right\rfloor$.

\item[{\rm (iv)}] If $r,t\ge 1$, then $\mathrm{dim}_I(K_{r,t})= r+t-2.$
\end{enumerate}
\end{proposition}

\begin{proof}
(i) Clearly $K_n$ is an edge-triangular graph. So, the result follows from Corollary \ref{corollary:edge-triangular}, and the fact that any set containing all but one vertex of any graph $G$ is an incidence generator for $G$.

(ii) Let $V(P_n)=\{v_0,v_1,\dots,v_{n-1}\}$ such that $v_iv_{i+1}\in E(P_n)$ for every $i\in\{0,\dots,n-2\}$. Consider the set $S'=\{v_i: i\ge 2\mbox{ and } i\equiv 0\mbox{ or }i\equiv 2\; (\mathrm{mod }\; 3)\}$. Note that any two edges of $P_n$ are incidently resolved by $S'$, and so $\mathrm{dim}_I(P_n)\leq |S'|=\left\lfloor\frac{2(n-1)}{3}\right\rfloor$.

On the other hand, let $S$ be an incidence basis for $P_n$. There could be at most one edge which is not incident to any vertex of $S$. Also, if $v_i\in S$ and $i\ne 0,n-1$, then $v_{i-1}\in S$ or $v_{i+1}\in S$. Thus, for any three consecutive vertices $v_{i},v_{i+1},v_{i+2}$ at least two of them are in $S$ with only one possible exception. According to these facts, $\mathrm{dim}_I(P_n)=|S|\ge \left\lfloor\frac{2(n-1)}{3}\right\rfloor$, which completes the proof of (ii).

(iii) Let $e=uv$ be any edge of $C_n$. Clearly, $C_n-e\cong P_n$. It is well known, see \cite{MeMo}, that $\rho (T)=\gamma (T)$ for every tree $T$ and we have $\rho (P_n)=\left\lceil \frac{n}{3}\right\rceil$. Moreover, there always exists a $\rho(P_n)$-set $P$ such that $u,v\in P$. Hence, $|P_e(C_n)|=\left\lceil \frac{n}{3}\right\rceil$. On the other hand, $\rho(C_n)=\left\lfloor \frac{n}{3}\right\rfloor$. If $n\neq 3k$, then $|P_e(C_n)|=\rho(C_n)+1$, and by Theorem \ref{theorem:n-k}, we have $\mathrm{dim}_I(C_n)=n-\left\lceil \frac{n}{3}\right\rceil=\left\lfloor \frac{2n}{3}\right\rfloor$. For $n=3k$ we have $|P_e(C_n)|=\rho(C_n)=k$ for every edge $e$ and, again by Theorem \ref{theorem:n-k}, we have $\mathrm{dim}_I(C_n)=n-\rho(C_n)=3k-k=\frac{2n}{3}=\left\lfloor \frac{2n}{3}\right\rfloor$. So, we are done with the proof of (iii).

(iv) Clearly $P_e(K_{r,t})=\{u,v\}$ for any edge $e=uv$ of $K_{r,t}$, while $\rho(K_{r,t})=1$. By Theorem \ref{theorem:n-k} we have $\mathrm{dim}_I(K_{r,t})=r+t-2$, because all edges are symmetric to each other.
\end{proof}

We end this section with a short discussion on incidence dimension of trees. We recall that $A\triangle B$ denotes the symmetric difference of the sets $A$ and $B$.

\begin{theorem}
\label{edge} Let $G$ be a graph. If $\mathrm{dim}_I(G)=|V(G)|-\rho(G)-1$,
then $G[P_1\triangle P_2]$ is not an empty graph for some maximum packings $P_1$ and $P_2$ of $G$.
\end{theorem}

\begin{proof}
Let $G$ be a graph and let $S$ be an incidence basis of $G$ such that $\mathrm{dim}_I(G)=|V(G)|-\rho(G)-1$.
Let $P=V(G)-S$. Since $|P|=\rho(G)+1$, $P$ is not a packing for $G$. So, there exist two different vertices $u, v \in P$, such that $1\leq \textnormal{d}(u,v) \leq 2$. If $\textnormal{d}(u,v)=2$, then the edges $uw$ and $wv$, where $w$ is a common neighbour of $u$ and $v$, are not resolved by $S$, a contradiction. Thus $\textnormal{d}(u,v)=1$. Let $P_1=V(G)-(S \cup \{u\})$ and $P_2=V(G)-(S \cup \{v\})$. The cardinality of both packings is maximum possible, since $|P_1|=|P_2|=|V(G)|-\left( (|V(G)|-\rho(G)-1)+1 \right) = \rho(G)$. Since $u$ and $v$ are adjacent it follows that $G[P_1\triangle P_2]$ is not an empty graph and the proof is completed.
\end{proof}

The converse implication of Theorem \ref{edge} does not hold in general as we can see from the example in Figure \ref{example-for-theorem-edge}. The left side and middle picture show two different maximum packings of $G$, which are denoted with black vertices. On the right side there is a picture of $G[P_1\triangle P_2]$ which is clearly not an empty graph. However, the incidence dimension of $G$ is not equal to $|V(G)|-\rho(G)-1$.

\begin{figure}[ht!] \label{example-for-theorem-edge}
\begin{center}
\begin{tikzpicture}[xscale=0.4, yscale=0.5, style=thick,x=1cm,y=1cm]
\tikzstyle{rn}=[circle,fill=white,draw, inner sep=0pt, minimum size=5pt]
\tikzstyle{bn}=[circle,fill=black,draw, inner sep=0pt, minimum size=5pt]		
	\node [style=rn] (a1) at (0, 0) {};
	\node [style=rn] (a2) at (2, 0) {};
	\node [style=bn] (a3) at (4, 0) {};
	\node [style=rn] (a4) at (6, 0) {};
	\node [style=rn] (a5) at (8, 0) {};
	\node [style=rn] (a6) at (10, 0) {};
	\node [style=rn] (a7) at (12, 0) {};
	
	\node [style=rn] (b1) at (0, 2) {};
	\node [style=rn] (b2) at (6, 2) {};
	\node [style=rn] (b3) at (12, 2) {};
	
	\node [style=bn] (c1) at (0, 4) {};
	\node [style=bn] (c2) at (6, 4) {};
	\node [style=bn] (c3) at (12, 4) {};
	
	\node [style=bn] (d1) at (0, -2) {};
	\node [style=bn] (d2) at (12, -2) {};
	
	\draw (a1) -- (a2);
	\draw (a2) -- (a3);
	\draw (a3) -- (a4);
	\draw (a4) -- (a5);
	\draw (a5) -- (a6);
	\draw (a6) -- (a7);

	\draw (a1) -- (b1);
	\draw (a4) -- (b2);
	\draw (a7) -- (b3);

	\draw (b1) -- (c1);
	\draw (b2) -- (c2);
	\draw (b3) -- (c3);

	\draw (a1) -- (d1);
	\draw (a7) -- (d2);
	
	\node [style=rn] (a1) at (14, 0) {};
	\node [style=bn] (a2) at (16, 0) {};
	\node [style=rn] (a3) at (18, 0) {};
	\node [style=rn] (a4) at (20, 0) {};
	\node [style=bn] (a5) at (22, 0) {};
	\node [style=rn] (a6) at (24, 0) {};
	\node [style=rn] (a7) at (26, 0) {};
	
	\node [style=rn] (b1) at (14, 2) {};
	\node [style=rn] (b2) at (20, 2) {};
	\node [style=rn] (b3) at (26, 2) {};
	
	\node [style=bn] (c1) at (14, 4) {};
	\node [style=bn] (c2) at (20, 4) {};
	\node [style=bn] (c3) at (26, 4) {};
	
	\node [style=rn] (d1) at (14, -2) {};
	\node [style=bn] (d2) at (26, -2) {};
	
	\draw (a1) -- (a2);
	\draw (a2) -- (a3);
	\draw (a3) -- (a4);
	\draw (a4) -- (a5);
	\draw (a5) -- (a6);
	\draw (a6) -- (a7);

	\draw (a1) -- (b1);
	\draw (a4) -- (b2);
	\draw (a7) -- (b3);

	\draw (b1) -- (c1);
	\draw (b2) -- (c2);
	\draw (b3) -- (c3);

	\draw (a1) -- (d1);
	\draw (a7) -- (d2);
	
	\node [style=rn] (a1) at (28, 0) {};
	\node [style=bn] (a2) at (30, 0) {};
	\node [style=bn] (a3) at (32, 0) {};
	\node [style=rn] (a4) at (34, 0) {};
	\node [style=bn] (a5) at (36, 0) {};
	\node [style=rn] (a6) at (38, 0) {};
	\node [style=rn] (a7) at (40, 0) {};
	
	\node [style=rn] (b1) at (28, 2) {};
	\node [style=rn] (b2) at (34, 2) {};
	\node [style=rn] (b3) at (40, 2) {};
	
	\node [style=rn] (c1) at (28, 4) {};
	\node [style=rn] (c2) at (34, 4) {};
	\node [style=rn] (c3) at (40, 4) {};
	
	\node [style=bn] (d1) at (28, -2) {};
	\node [style=rn] (d2) at (40, -2) {};
	
	\draw (a2) -- (a3);
\end{tikzpicture}
\end{center}
\caption{An example showing that the converse implication of Theorem \ref{edge} does not in general hold.}
\end{figure}
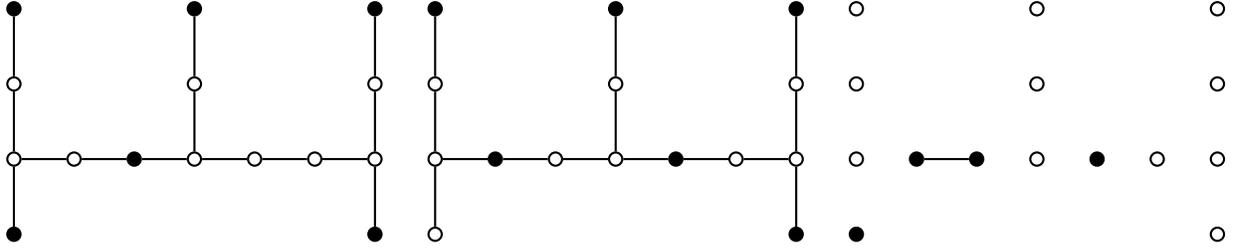

We next provide an exact result for the class of trees with a unique maximum packing. Two characterizations of such trees were presented recently in \cite{BoPe}.

\begin{theorem}
If $T$ is a tree with the unique maximum packing $P$, then $\mathrm{dim}_I(G)=|V(G)|-\rho(G)$.
\end{theorem}

\begin{proof}
Let $T$ be a tree and $P$ its unique maximum packing.
To prove that $\mathrm{dim}_I(G)=|V(G)|-\rho(G)$ we will use a contraposition of the Theorem~\ref{edge}:
\textit{If $G[P_1\triangle P_2]$ is an empty graph, then $\mathrm{dim}_I(G)\neq|V(G)|-\rho(G)-1$.}
It follows that $P_1=P_2=P$, because $P$ is a unique maximum packing of $T$. So, $G[P_1\triangle P_2]$ is an empty graph for any maximum packings $P_1$ and $P_2$ and $\mathrm{dim}_I(G)\neq|V(G)|-\rho(G)-1$. By using Theorem~\ref{theorem:n-k}, we conclude that $\mathrm{dim}_I(G)=|V(G)|-\rho(G)$.
\end{proof}


\section{Complexity of the problem}

In this section we consider the computational complexity of the problem of computing the incidence dimension of a graph. It is well known that the problem of calculating the metric dimension of a graph is NP-hard as stated in the book \cite{garey}, and formally proved in \cite{Khuller1996}. We show that the problem of finding the incidence dimension of an arbitrary graph is also NP-hard. For this, we strongly rely on edge-triangular graphs. We first consider the following decision problem.

\begin{equation*}
\begin{tabular}{|l|}
\hline
\mbox{INCIDENCE DIMENSION PROBLEM (IDIM problem for short)} \\
\mbox{INSTANCE: A graph $G$ of order $n\ge 3$ and an integer $1\le r\le
n-1$.} \\
\mbox{QUESTION: Is $\mathrm{dim}_I(G)\le r$?} \\ \hline
\end{tabular}%
\end{equation*}%
\newline

To study the complexity of IDIM problem we make a reduction from the 3-SAT problem, which is one of the most classical NP-complete problems known in the literature. For more information on the 3-SAT problem and reducibility of NP-complete problems in general, we suggest \cite{garey}.

\begin{theorem}\label{theorem:NP-complete}
The IDIM problem is NP-complete.
\end{theorem}

\begin{proof}
For a set of vertices $S$ guessed for the problem by a nondeterministic algorithm, one needs to iterate through all pairs of edges and check that every pair is incidently resolved by one vertex from $S$. This can be done in polynomial time and therefore IDIM problem is in NP class.

We make a polynomial transformation of 3-SAT problem to IDIM problem in the following way.
Consider an arbitrary instance of the 3-SAT problem, \emph{i.e.}, a finite set $U=\{u_1,\ldots,u_n\}$ of Boolean variables and a collection $C=\{c_1, \ldots, c_m\}$ of clauses over those Boolean variables.
We will construct a graph $G=(V,E)$ and set a positive integer $r \leq |V|-1$, such that $\mathrm{dim}_I(G) \leq r$ if and only if $C$ is satisfiable.
The construction will be made up of several gadgets and edges between them.

For each variable $u_i \in U$, $1\leq i\leq n$, we construct a truth-setting gadget $X_i=(V_i,E_i)$, with $V_i=\{x_i,y_i,z_i,w_i,T_i,F_i\}$ and $E_i=\{x_iy_i,x_iz_i,y_iz_i,y_iw_i,z_iw_i,w_iT_i,w_iF_i,T_iF_i\}$, see Figure \ref{figure:NP1}.
Each truth-setting gadget is connected with the rest of the graph only through $T_i$ and $F_i$ nodes, which are \texttt{TRUE} and \texttt{FALSE} representing values, respectively.

\begin{claim}\label{remark:NP1}
Let $u_i$ be an arbitrary variable in $U$. Any incidence generator $S$ must contain at least four vertices from its truth-setting gadget. Moreover, if there are exactly four vertices from a truth-setting gadget in $S$, then $y_i,w_i,z_i\in S$ and $x_i\notin S$.
\end{claim}

\proof Towards a contradiction suppose that there exists an incidence generator $S$ with less than four vertices from the truth-setting gadget corresponding to $u_i$. It follows that there exists a set of three vertices $W_i=\{v_1,v_2,v_3\} \subset V_i$ that are not in $S$. Make a partition of $V_i$ into two sets $V_i=\{x_i,y_i,z_i \} \cup \{ w_i, T_i, F_i\}$. There are at least two vertices from $W_i$ in one of the partition sets.
Since each partition set forms a triangle, it follows, that there is an edge lying outside $S$. That is a contradiction with Proposition \ref{proposition:edge-triangular}.

Suppose now that exactly four vertices from a truth-setting gadget are in $S$. If $w_i\notin S$, then $T_i,F_i,y_i,z_i\in S$ because otherwise we have a contradiction with Proposition \ref{proposition:edge-triangular}. But then $x_i\notin S$ and edges $w_iz_i$ and $x_iz_i$ are not distinguished by $S$, a contradiction. Hence, $w_i\in S$. If $y_i\notin S$ (resp. $z_i\notin S$), then $x_i\in S$ and $z_i\in S$ (resp. $y_i\in S$) to fulfill Proposition \ref{proposition:edge-triangular}. Clearly, exactly one from $T_i$ and $F_i$ is in $S$. If $T_i\in S$, then edges $F_iw_i$ and $y_iw_i$ (resp. $z_iw_i$) are not distinguished by $S$, a contradiction. Thus, $y_i,z_i\in S$. If in addition $x_i\in S$, then for the triangle $w_iT_iF_iw_i$ we have a contradiction with Proposition \ref{proposition:edge-triangular}. Therefore, $x_i\notin S$.~\smallqed
\medskip

\begin{figure}[h]
\centering
\begin{tikzpicture}[scale=.9, transform shape]

\node [draw, shape=circle] (a1) at (0.2,0) {};
\node [draw, shape=circle] (a2) at (1.8,0) {};
\node [draw, shape=circle] (a3) at (1,1) {};
\node [draw, shape=circle] (a4) at (0.2,2) {};
\node [draw, shape=circle] (a5) at (1.8,2) {};
\node [draw, shape=circle] (a6) at (1,3) {};

\draw (0.1,0) node[left] {$T_i$};
\draw (1.9,0) node[right] {$F_i$};
\draw (0.9,1) node[left] {$w_i$};
\draw (0.1,2) node[left] {$y_i$};
\draw (1.9,2) node[right] {$z_i$};
\draw (1,3.1) node[above] {$x_i$};

\foreach \from/\to in {
 a1/a2, a1/a3, a2/a3, a3/a4, a3/a5, a4/a5, a4/a6, a5/a6}
\draw (\from) -- (\to);
\end{tikzpicture}
\caption{The truth-setting gadget for variable $u_i$.}\label{figure:NP1}
\end{figure}
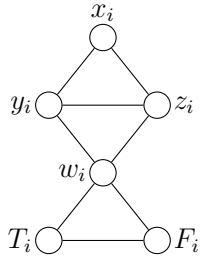

For each clause $c_j=y_j^1 \vee y_j^2 \vee y_j^3$, $1\leq j\leq m_j$, where $y_j^k$ is a literal in the clause $c_j$, we construct a satisfaction testing gadget $Y_j=(V_j',E_j')$, with $V_j'=\{a_j^1,b_j^1,c_j^1,a_j^2,b_j^2,c_j^2, a_j^3,b_j^3,c_j^3\}$ and
$E_j'=\{a_j^1b_j^1,a_j^1c_j^1,b_j^1c_j^1,a_j^2b_j^2,a_j^2c_j^2,b_j^2c_j^2,a_j^3b_j^3,a_j^3c_j^3,
b_j^3c_j^3,a_j^1a_j^2,a_j^2a_j^3,a_j^3a_j^1,b_j^1b_j^2,b_j^2b_j^3,b_j^3b_j^1,c_j^1c_j^2,
c_j^2c_j^3,c_j^3c_j^1\}$ (see Figure \ref{figure:NP2}).
Notice that the satisfaction testing gadget is isomorphic to the Cartesian product$C_3 \square C_3$.

\begin{claim}\label{remark:NP2}
Let $c_j$ be an arbitrary clause in $C$ and $Y_j=(V_j',E_j')$ its satisfaction testing gadget. Then any incidence generator must contain at least 8 vertices from $V_j'$.
\end{claim}

\proof
Towards a contradiction, suppose there exists an incidence generator $S$ with less than 8 vertices from the satisfaction testing gadget corresponding to $c_j$. It follows that there exist two vertices $x,y \in V_j'$ that are not in $S$. Since the diameter of $C_3 \square C_3 $ is 2, the vertices $x$ and $y$ are either at distance one or two. Every edge of $Y_j$ is a part of some triangle, and so $x$ and $y$ cannot be at distance one, due to Proposition \ref{proposition:edge-triangular}. Thus, there is a vertex $z$ such that the edges $xz$ and $zy$ exist. But those two edges are not resolved by any endpoint, a contradiction.~\smallqed
\medskip

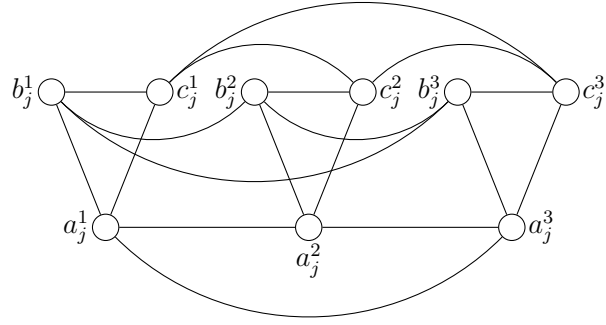
\begin{figure}[h]
\centering
\begin{tikzpicture}[scale=.9, transform shape]

\node [draw, shape=circle] (a1) at (-3,-2) {};
\node [draw, shape=circle] (b1) at (-3.8,0) {};
\node [draw, shape=circle] (c1) at (-2.2,0) {};
\node [draw, shape=circle] (a2) at (0,-2) {};
\node [draw, shape=circle] (b2) at (-0.8,0) {};
\node [draw, shape=circle] (c2) at (0.8,0) {};
\node [draw, shape=circle] (a3) at (3,-2) {};
\node [draw, shape=circle] (b3) at (2.2,0) {};
\node [draw, shape=circle] (c3) at (3.8,0) {};

\draw (-3.1,-2) node[left] {$a_j^1$};
\draw (-3.9,0) node[left] {$b_j^1$};
\draw (-2.1,0) node[right] {$c_j^1$};
\draw (0,-2.1) node[below] {$a_j^2$};
\draw (-0.9,0) node[left] {$b_j^2$};
\draw (0.9,0) node[right] {$c_j^2$};
\draw (3.1,-2) node[right] {$a_j^3$};
\draw (2.1,0) node[left] {$b_j^3$};
\draw (3.9,0) node[right] {$c_j^3$};


\foreach \from/\to in {
 a1/b1, b1/c1, c1/a1, a2/b2, b2/c2, c2/a2, a3/b3, b3/c3, c3/a3,
 a1/a2, a2/a3}
\draw (\from) -- (\to);
\foreach \from/\to in {
    c1/c2, c2/c3, c1/c3}
\draw (\from) edge[out=45, in=135]  (\to);
\foreach \from/\to in {
   a1/a3,  b1/b2, b2/b3, b1/b3}
\draw (\from) edge[out=-45, in=-135]  (\to);
\end{tikzpicture}
\caption{The satisfaction testing gadget for clause $c_j$.}\label{figure:NP2}
\end{figure}

We also add some edges to connect the truth-setting gadgets with corresponding satisfaction testing gadgets. If a variable $u_i$ occurs as a literal $y_j^k$ in a clause $c_j=y_j^1 \vee y_j^2 \vee y_j^3$, then we add the following edges.
If $y_j^k$ is a positive literal then we add the edges $F_ib_j^k$ and $F_ic_j^k$. If a variable $y_j^k$ is a negative literal in a clause $c_j$, then we add the edges $T_ib_j^k$ and $T_ic_j^k$.
For each clause $c_j \in C$ denote those six added edges with $E_j''$. We call them \textit{communication} edges. Figure \ref{figure:NP3} shows the edges that were added corresponding to the clause $c_j= (\overline{u_1} \vee \overline{u_2} \vee u_3)$, where $\overline{u_1}$ and $\overline{u_2}$ represent the negative literal corresponding to the variables $u_1$ and $u_2$, respectively.

\begin{figure}[h]
\centering
\begin{tikzpicture}[scale=.9, transform shape]

\node [draw, shape=circle] (a11) at (-4.8,2) {};
\node [draw, shape=circle] (a12) at (-3.2,2) {};
\node [draw, shape=circle] (a13) at (-4,3) {};
\node [draw, shape=circle] (a14) at (-4.8,4) {};
\node [draw, shape=circle] (a15) at (-3.2,4) {};
\node [draw, shape=circle] (a16) at (-4,5) {};

\draw (-4.9,2) node[left] {$T_1$};
\draw (-3.1,2) node[right] {$F_1$};
\draw (-4.1,3) node[left] {$w_1$};
\draw (-4.9,4) node[left] {$y_1$};
\draw (-3.1,4) node[right] {$z_1$};
\draw (-4,5.1) node[above] {$x_1$};

\foreach \from/\to in {
 a11/a12, a11/a13, a12/a13, a13/a14, a13/a15, a14/a15, a14/a16, a15/a16}
\draw (\from) -- (\to);

\node [draw, shape=circle] (a21) at (-0.8,2) {};
\node [draw, shape=circle] (a22) at (0.8,2) {};
\node [draw, shape=circle] (a23) at (0,3) {};
\node [draw, shape=circle] (a24) at (-0.8,4) {};
\node [draw, shape=circle] (a25) at (0.8,4) {};
\node [draw, shape=circle] (a26) at (0,5) {};

\draw (-0.9,2) node[left] {$T_2$};
\draw (0.9,2) node[right] {$F_2$};
\draw (-0.1,3) node[left] {$w_2$};
\draw (-0.9,4) node[left] {$y_2$};
\draw (0.9,4) node[right] {$z_2$};
\draw (0,5.1) node[above] {$x_2$};

\foreach \from/\to in {
 a21/a22, a21/a23, a22/a23, a23/a24, a23/a25, a24/a25, a24/a26, a25/a26}
\draw (\from) -- (\to);

\node [draw, shape=circle] (a31) at (3.2,2) {};
\node [draw, shape=circle] (a32) at (4.8,2) {};
\node [draw, shape=circle] (a33) at (4,3) {};
\node [draw, shape=circle] (a34) at (3.2,4) {};
\node [draw, shape=circle] (a35) at (4.8,4) {};
\node [draw, shape=circle] (a36) at (4,5) {};

\draw (3.1,2) node[left] {$T_3$};
\draw (4.9,2) node[right] {$F_3$};
\draw (3.9,3) node[left] {$w_3$};
\draw (3.1,4) node[left] {$y_3$};
\draw (4.9,4) node[right] {$z_3$};
\draw (4,5.1) node[above] {$x_3$};

\foreach \from/\to in {
 a31/a32, a31/a33, a32/a33, a33/a34, a33/a35, a34/a35, a34/a36, a35/a36}
\draw (\from) -- (\to);

\node [draw, shape=circle] (a1) at (-3,-2) {};
\node [draw, shape=circle] (b1) at (-3.8,0) {};
\node [draw, shape=circle] (c1) at (-2.2,0) {};
\node [draw, shape=circle] (a2) at (0,-2) {};
\node [draw, shape=circle] (b2) at (-0.8,0) {};
\node [draw, shape=circle] (c2) at (0.8,0) {};
\node [draw, shape=circle] (a3) at (3,-2) {};
\node [draw, shape=circle] (b3) at (2.2,0) {};
\node [draw, shape=circle] (c3) at (3.8,0) {};

\draw (-3.1,-2) node[left] {$a_j^1$};
\draw (-3.9,0) node[left] {$b_j^1$};
\draw (-2.1,0) node[right] {$c_j^1$};
\draw (0,-2.1) node[below] {$a_j^2$};
\draw (-0.9,0) node[left] {$b_j^2$};
\draw (0.9,0) node[right] {$c_j^2$};
\draw (3.1,-2) node[right] {$a_j^3$};
\draw (2.1,0) node[left] {$b_j^3$};
\draw (3.9,0) node[right] {$c_j^3$};


\foreach \from/\to in {
 a1/b1, b1/c1, c1/a1, a2/b2, b2/c2, c2/a2, a3/b3, b3/c3, c3/a3,
 a1/a2, a2/a3, a11/b1, a11/c1, a21/b2, a21/c2, a32/b3, a32/c3}
\draw (\from) -- (\to);
\foreach \from/\to in {
    c1/c2, c2/c3, c1/c3}
\draw (\from) edge[out=45, in=135]  (\to);
\foreach \from/\to in {
   a1/a3,  b1/b2, b2/b3, b1/b3}
\draw (\from) edge[out=-45, in=-135]  (\to);

\end{tikzpicture}
\caption{The subgraph associated to the clause $c_j=(\overline{u_1} \vee \overline{u_2} \vee u_3)$.}\label{figure:NP3}
\end{figure}
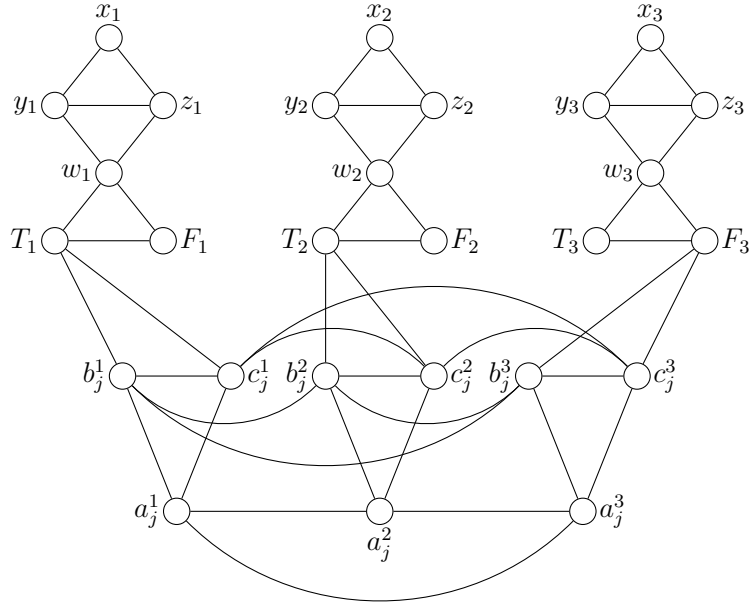

The construction of the IDIM instance is then completed by setting $r=4n+8m$ and $G=(V,E)$, where
$$ V= \left( \bigcup_{i=1}^n V_i \right) \cup \left( \bigcup_{j=1}^m V_j' \right)$$
and
$$ E= \left( \bigcup_{i=1}^n E_i \right) \cup \left( \bigcup_{j=1}^m E_j' \right) \cup \left( \bigcup_{j=1}^m E_j'' \right) $$

One can do the described construction in polynomial time. Notice that the graph $G$ is edge-triangular.

If we show that $C$ is satisfiable if and only if $G$ has incidence dimension less or equal than $r$, then the proof of NP-completeness is completed.
From Claims \ref{remark:NP1} and \ref{remark:NP2} we get the following corollary.

\begin{corollary}\label{corollary:np}
The incidence dimension of the graph $G$ constructed above is at least $r=4n+8m$.
\end{corollary}

The following Lemmas together with Corollary \ref{corollary:np} complete the proof for the IDIM problem being NP-complete.

\begin{lemma}\label{lema:NP1}
If $C$ is satisfiable, then $\mathrm{dim}_I(G)=r$.
\end{lemma}

\proof We construct an incidence generator $S$ of size $r$ based on a truth assignment of elements from the set $U$ that satisfies the collection of clauses $C$.
Let $t:U \rightarrow \{\texttt{TRUE,FALSE}\}$ be a truth assignment that satisfies the collection of clauses $C$.
For each clause $c_j=y_j^1 \vee y_j^2 \vee y_j^3$, from $C$, put into $S$ the vertices $b_j^k,c_j^k$ for $k \in \{1,2,3\}$. Since the collection of clauses $C$ is satisfiable, there exists a literal $y_j^k$, $k \in \{1,2,3\}$, that satisfies $c_j$. Fix one such $k$ and put into $S$ the other two vertices $a_j^\ell$ for $\ell \in \{1,2,3\} \setminus \{k\}$

For each Boolean variable $u_i \in U$, put into $S$ the vertices $\{y_i,z_i,w_i\}$. Also add to the set $S$, the vertex $F_i$ if $t(u_i)=\texttt{TRUE}$, or the vertex $T_i$ if $t(u_i)=\texttt{FALSE}$. The cardinality of the constructed set $S$ is clearly $r=4n+8m$.

We now take a look at the set $X=V(G) \setminus S$. For each $u_i \in U$ there are $x_i$ and exactly one of the vertices from the set $\{T_i, F_i\}$ in the set $X$. The distance between these two vertices is three.
For each $c_j \in C$ exactly one of the vertices $a_j^1,a_j^2,a_j^3$ is in $X$. The vertex that is in $X$ corresponds to the variable that satisfies $c_j$. It follows that this vertex is at the distance three or more from all the other vertices in $X$. All other possible pairs of vertices in $X$ are also at distance greater or equal to three. It follows that $X$ is a packing of $G$. Graph $G$ is edge-triangular, and together with the Corollary \ref{corollary:edge-triangular}, it follows that $S$ is an incidence generator for $G$.~\smallqed
\medskip

\begin{lemma}\label{lema:NP2}
If $\mathrm{dim}_I(G)=r$, then the collection of clauses $C$ is satisfiable.
\end{lemma}

\proof Let $S$ be an arbitrary incidence generator for $G$ with cardinality $r$. The set $S$ must contain at least eight vertices from each satisfaction testing gadget and at least four vertices from each truth-setting gadget due to Claims \ref{remark:NP1} and \ref{remark:NP2}. Since $|S|=r=8m+4n$, it follows that in $S$ there are exactly four vertices from each truth-setting component, and exactly eight vertices from each satisfaction testing component.
Since the graph $G$ is edge-triangular, and together with Corollary \ref{corollary:edge-triangular}, it follows that $X=V(G) \setminus S$ is a packing for $G$.
Moreover, for each $i\in \{1, \ldots, n\}$ it holds that $x_i\in X$ and exactly one of the vertices $T_i, F_i$ is in $X$, by Claim \ref{remark:NP1}. For each $j \in \{1, \ldots, m\}$ exactly one of the vertices $a_j^1, a_j^2, a_j^3$ is in $X$ because $b_j^i$ and $c_j^i$, $i\in \{1, 2,3\}$, are in a common triangle with either $T_\ell$ or $F_\ell$ where $u_\ell$ belongs to clause $c_j$.

We now define a function that satisfies all clauses from $C$. For an arbitrary $i\in \{1, \ldots, n\}$, let $v_i \in \{T_i,F_i\} \cap X$. Let $t: U \rightarrow \{\texttt{TRUE,FALSE}\}$ be as follows:
$$t(u_i)=\left\{\begin{array}{ll}
                           \texttt{TRUE}, & v_i=T_i , \\
                           \texttt{FALSE}, & v_i=F_i .
                         \end{array}
\right.$$
We need to show that $t$ is a satisfying truth assignment for $C$.
Let $c_j=y_j^1 \vee y_j^2 \vee y_j^3 \in C$ be an arbitrary clause and denote the corresponding boolean variables with $u_{j_1},u_{j_2},u_{j_3}$, respectively.
To show that at least one of its literals has value \texttt{TRUE}, take the vertex from $V_j'$ that belongs to $X$. There is exactly one of the vertices $a_j^1, a_j^2, a_j^3$ in $X$. Let $a_j^k$, $k \in \{1,2,3\}$, be the vertex that is in $X$.
The communication edges are added in such a way that $a_j^k$ can be in $X$ (packing set) only if $u_{j_k}$ occurs in $c_j$ as:
\begin{itemize}
\item a positive literal and $v_{j_k} =T_{j_k}$;
\item a negative literal and $v_{j_k} =F_{j_k}$.
\end{itemize}
In both cases $c_j$ is satisfied by the literal corresponding to the variable $u_{j_k}$. It finally follows that $C$ is satisfiable, which completes the proof of this lemma.~\smallqed
\medskip

Lemmas \ref{lema:NP1} and \ref{lema:NP2} show that the above construction is a polynomial transformation from 3-SAT to the IDIM problem. Therefore, the proof of Theorem \ref{theorem:NP-complete} is completed.
\end{proof}

The proof of Theorem \ref{theorem:NP-complete} yields the following result.

\begin{corollary}
\label{np-hard} The problem of finding the incidence dimension of a
graph is NP-hard.
\end{corollary}


\section{Some final remarks on $\mathrm{dim}_I(G)$}

Given any graph $G$, it is easy to see that the set $V(G)$ minus one
arbitrary vertex is an incident generator for $G$. On the other hand, given
an incidence basis for $G$, for all but probably one edge in $E(G)$ at least
one of its endpoints belongs to $S$. Moreover for any three edges incident
with a same vertex, at least two different endpoints of two different edges
must be in $S$ too. In consequence, the following bounds are easy to deduce.

\begin{remark}
\label{trivial-bounds} If $G$ is a connected graph of order $n$ with at least two edges, then
$$\left\lfloor\frac{n}{2}\right\rfloor\le \mathrm{dim}_I(G)\le n-1.$$
\end{remark}

The lower bound of Remark~\ref{trivial-bounds} is achieved for a path $P_n$, $n\in\{3,4,5,6,8\}$, a cycle $C_4$, a star $K_{1,3}$ and some graphs obtained by attaching a pendant vertex or an edge to some vertices of previously mentioned examples. While it is not clear if this list is complete, we can entirely describe all graphs achieving the upper bound of Remark \ref{trivial-bounds}.

\begin{proposition}
Let $G$ be a connected graph of order $n$ with at least two edges.
Then $\mathrm{dim}_I(G)=n-1$ if and only if any two vertices of $G$ have a
common neighbor.
\end{proposition}

\begin{proof}
If there are two different vertices $x,y\in V(G)$ such that they do not have a common neighbor, then it is not difficult to see that the set $V(G)-\{x,y\}$ is an incidence generator for $G$. Thus, to have an incidence generator of order $n-1$, it is required that any two different vertices of $G$ have a common neighbor, and vice versa.
\end{proof}

Now, concerning the bounds of Remark \ref{trivial-bounds}, we next study the
existence of graphs $G$ of order $n$ and incidence dimension $r$ for any $r,n
$ such that $\left\lfloor\frac{n}{2}\right\rfloor\le r\le n-1$.

\begin{proposition}
For any integers $r,n$ with $2\le \left\lfloor\frac{n}{2}\right\rfloor\le
r\le n-1$ there exists a graph $G$ of order $n$ such that $\mathrm{dim}%
_I(G)=r$.
\end{proposition}

\begin{proof}
If $r=2$, then $n\in\{4,5\}$. In these situations, the graphs $P_4$ and $P_5$ respectively satisfy our requirements. Hence, from now on we may assume $r\ge 3$.

Consider $n$ is odd and $r=\left\lfloor\frac{n}{2}\right\rfloor < n-1$. Let $G_{r,n}$ be the graph obtained as follows.
\begin{itemize}
  \item We begin with a complete graph $K_r$ with vertex set $V=\{u_1,\ldots,u_r\}$.
  \item Add $r+1$ vertices $w,v_1,\ldots,v_r$.
  \item Add the edges $u_iv_i$ for every $i\in \{1,\dots,r\}$ and the edge $wv_1$.
\end{itemize}
Clearly, $G_{r,n}$ has order $2r+1=n$. It is not difficult to see that $V$ is an incidence generator for $G_{r,n}$ and so, $\mathrm{dim}_I(G_{r,n})\le r$. Now, suppose $\mathrm{dim}_I(G_{r,n}) < r$ and let $S$ be an incidence basis for $G_{r,n}$.
That means that there is at least one vertex $u_j\in V$ such that $u_j\notin S$. If there exists some other vertex $u_i\in V$, $i\ne j$, such that $u_i\notin S$, then there are two edges $u_ju_k$, $u_iu_k$, with $k\ne i,j$ (since $r\ge 3$), such that they are not incidently resolved by $S$, a contradiction. Thus, $V-\{u_j\}\subseteq S$, which means  $|S| = r-1$ and $S=V-\{u_j\}$. But, in such a case, the edges $wv_1$ and $u_jv_j$ are not incidently resolved by $S$, which is a contradiction again. As a consequence, $\mathrm{dim}_I(G_{r,n})=r$.

We next consider ($n$ is even and $r=\left\lfloor\frac{n}{2}\right\rfloor < n-1$) or $\left\lfloor\frac{n}{2}\right\rfloor < r < n-1$. Let $G'_{r,n}$ be the graph obtained as follows.
\begin{itemize}
  \item We begin with a complete graph $K_r$ with vertex set $V=\{u_1,\ldots,u_r\}$.
  \item Add $n-r$ vertices $v_1,\ldots,v_{n-r}$.
  \item Add the edges $u_iv_i$ for every $i\in \{1,\ldots,n-r-1\}$.
  \item Add the edges $v_{n-r}u_{i}$ for every $i\in \{n-r,\ldots,r\}$.
  \item Add the edge $v_1v_2$ (notice that such two vertices always exist because $r<n-1$).
\end{itemize}
Clearly the order of $G'_{r,n}$ is $n$ and we can easily notice that $V$ is an incidence generator for $G'_{r,n}$ and so, $\mathrm{dim}_I(G_{r,n})\le r$. Hence, suppose $\mathrm{dim}_I(G_{r,n}) < r$ and let $S'$ be an incidence basis for $G'_{r,n}$. In consequence, there is at least one vertex $u_j\in V$ such that $u_j\notin S'$. A similar procedure as earlier leads to the fact that $\mathrm{dim}_I(G_{r,n}) = r-1$ and that $S'=V-\{u_j\}$. However, in this case there is an edge $u_jv_l$ for some $l\in \{1,\dots,r\}$ such that the edges $v_1v_2$ and $u_jv_l$ are not incidently resolved by $S'$, a contradiction. Therefore, $\mathrm{dim}_I(G_{r,n}) = r$.

We finally consider the situation $r = n-1$, which is straightforward to realize by just taking the complete graph $K_n$, and that completes the proof.
\end{proof}

It is natural to think that the incidence dimension is related to the (edge, or adjacency) dimension of graphs. Accordingly, we conclude this work by comparing $\mathrm{dim}_I(G)$ with $\mathrm{dim_e}(G)$ and $\mathrm{dim}_A(G)$.

\begin{proposition}\label{inc-adj-edge}
For any graph $G$ without isolated vertices, $\mathrm{dim}_I(G)\ge \max\{%
\mathrm{dim}_A(G),\mathrm{dim_e}(G)\}$.
\end{proposition}

\begin{proof}
  Let $S$ be an incidence basis for $G$. Consider two different vertices $x,y\in V(G)-S$. If $N(x)\cap S=\emptyset$ and $N(y)\cap S=\emptyset$, then since $G$ has no isolated vertices there are at least two edges $xx'$ and $yy'$ such that $x',y'\notin S$. Thus,  $xx'$ and $yy'$ are not incidently resolved by any vertex of $S$, which is a contradiction. So $N(x)\cap S\ne \emptyset$ or $N(y)\cap S\ne\emptyset$. Now, suppose $N(x)\cap S=N(y)\cap S$. Hence, there exists a vertex $w\in S$ such that the edges $xw$ and $yw$ are not incidently resolved by any vertex of $S$, a contradiction again. Thus, $N(x)\cap S\ne N(y)\cap S$ and, as a consequence, $S$ is an adjacency generator for $G$ and $\mathrm{dim}_I(G)\ge \mathrm{dim}_A(G)$.

  Now, since any two edges $e_1,e_2$ are incident to at least two different vertices $x,y$, and at least one of $x,y$ must be in $S$, it is clear that the edges $e_1,e_2$ are distinguished by $x$ or $y$. So, $S$ is also an edge metric generator for $G$, and $\mathrm{dim}_I(G)\ge \mathrm{dim_e}(G)$, which completes the proof.
\end{proof}

From \cite{JanOmo2012} we know that $\mathrm{dim}_A(K_{r,t})=r+t-2$. Also, from \cite{KeTrYe}, we have $\mathrm{dim_e}(K_{r,t})=r+t-2$. Now, from Proposition \ref{basic-values} (iv), we observe that the bound of Proposition \ref{inc-adj-edge} is tight. In such case, we have the equality $\mathrm{dim}_I(K_{r,t})=\mathrm{dim}_A(K_{r,t})=\mathrm{dim_e}(K_{r,t})$. An interesting problem is then to characterize the families of graphs for which the bound of Proposition \ref{inc-adj-edge} is achieved, and moreover, finding whether the situations $\mathrm{dim}_I(K_{r,t})=\mathrm{dim}_A(K_{r,t})\ne \mathrm{dim_e}(K_{r,t})$, $\mathrm{dim}_I(K_{r,t})=\mathrm{dim_e}(K_{r,t})\ne\mathrm{dim}_A(K_{r,t})$ or $\mathrm{dim}_I(K_{r,t})=\mathrm{dim}_A(K_{r,t})=\mathrm{dim_e}(K_{r,t})$ happen.

\section*{Acknowledgements}
Aleksander Kelenc and Iztok Peterin are partially supported by the Slovenian Research Agency by the projects No. N1-0063 and No. P1-0297, respectively, and both by project No. J1-9109.



\begin{thebibliography}{99}
\bibitem{biggs-1973} N.~Biggs, Perfect codes in graphs, J.\ Combin.\ Theory
Ser.\ B 15 (1973) 289--296.

\bibitem{BoPe} D. Bo\v{z}ovi\'c, I. Peterin, Graphs with unique maximum
packing of closed neighborhoods, submitted.\newline
https://mp.feri.um.si/osebne/peterin/clanki/Unique\%20packing\%20of%
\%20trees1.pdf

\bibitem{DoHiLe} M. P. Dobson, E. Hinrichsen, V. Leoni, Generalized limited
packings of some graphs with a limited number of $P_4$ partners, Theoret.
Comput. Sci. 579 (2015) 1--8.

\bibitem{Dutt} R. D. Dutton, Global domination and packing numbers, Ars
Combin. 101 (2011) 489--501.

\bibitem{adjacency1} A. Estrada-Moreno, Y. Ram\'irez-Cruz, J. A.
Rodr\'iguez-Vel\'azquez, On the adjacency dimension of graphs, Appl. Anal. Discrete Math. 10 (2016) 102--127.

\bibitem{adjacency2} H. Fernau, J. A. Rodr\'iguez-Vel\'azquez, On the
(adjacency) metric dimension of corona and strong product graphs and their
local variants: combinatorial and computational results, Discrete Appl. Math. 236 (2018) 183--202.

\bibitem{GaZv} A. Gagarin, V. Zverovich, The probabilistic approach to
limited packings in graphs, Discrete Appl. Math. 184 (2015) 146--153.

\bibitem{GaGuHa} R. Gallant, G. Gunther, B. Hartnell, D. F. Rall, Limited
packings in graphs, Discrete Appl. Math. 158 (2010) 1357--1364.

\bibitem{garey} M. R. Garey and D. S. Johnson,\emph{\ Computers and
Intractability: A Guide to the Theory of NP-Completeness},  W. H. Freeman \&
Co., New York, USA, 1979.

\bibitem{geneson} J. Geneson, Metric dimension and pattern avoidance in graphs, manuscript, (2018) arXiv:1807.08334 [math.CO]

\bibitem{Harary1976}
F.~Harary, R.~A. Melter, On the metric dimension of a graph, Ars Combin. 2   (1976) 191--195.

\bibitem{hhs1}  T.W.Haynes, S.T.Hedetniemi, P.J.Slater, Domination in Graphs, Marcel Dekker, Inc., New York, NY, 1998.

\bibitem{hhs2} T.W.Haynes, S.T.Hedetniemi, P.J.Slater, Domination in Graphs: Advanced Topics. Marcel Dekker, Inc., New York, NY, 1998.

\bibitem{HeLoRa} M. A. Henning, C. L\"owenstein, D. Rautenbach, Dominating
sets, packings, and the maximum degree, Discrete Math. 311 (2011) 2031--2036.

\bibitem{JanOmo2012} M.~Jannesari, B.~Omoomi, The metric dimension of the
lexicographic product of  graphs, Discrete Math. 312 (2012) 3349--3356.

\bibitem{JSRz} K. Junosza-Szaniawski, P. Rz\k{a}\.{z}ewski, On the number of
2-packings in a connected graph, Discrete Math. 312 (2012) 3444--3450.

\bibitem{mix-dim} A. Kelenc, D. Kuziak, A. Taranenko, I. G. Yero, On the mixed metric dimension of graphs, Appl. Math. Comput. 314 (2017) 429--438.

\bibitem{KeTrYe} A. Kelenc, N. Tratnik, I. G. Yero, Uniquely identifying the edges of a graph: the edge metric dimension, Discrete Appl. Math. in press (2018)

\bibitem{KPY} S. Klav\v{z}ar, I. Peterin, I. G. Yero, Graphs that are
simultaneously efficient open domination and efficient closed domination
graphs, Discrete Appl. Math. 217 (2017) 613--621.

\bibitem{kratica} J. Kratica, V. Filipovic, A. Kartelj, Edge metric dimension of some generalized Petersen graphs, manuscript, (2017) arXiv:1807.00580 [math.CO]

\bibitem{Khuller1996} S.~Khuller, B.~Raghavachari, A.~Rosenfeld, Landmarks in graphs, Discrete
  Appl. Math. 70 (3) (1996) 217--229.

\bibitem{MeMo} A. Meir, J. W. Moon, Relations between packing and covering
numbers of a tree, Pacific J. Math. 61 (1975) 225--233.

\bibitem{MSKG} D. A. Mojdeh, B. Samadi, A. Khodkar, H. R. Golmohammadi, On
the packing numbers in graphs, Australas. J. Combin. 71 (2018) 468--475.

\bibitem{peterin-yero} I. Peterin, I. G. Yero, Edge metric dimension of some graph operations, manuscript, (2018) arXiv:1809.08900 [math.CO]

\bibitem{SHSa} I. Sahul Hamid, S. Saravanakumar, Packing parameters in
graphs, Discuss. Math. Graph Theory 35 (2015) 5--6.

\bibitem{Slater1975}
P.~J. Slater, Leaves of trees, Congressus Numerantium 14 (1975) 549--559.

\bibitem{zubri} N. Zubrilina, On the edge dimension of a graph, Discrete Math. 341 (2018) 2083--2088.

\end{thebibliography}
\end{document}